\documentclass[a4paper,12pt]{extarticle}
\usepackage{latexsym}
\usepackage{amscd}
\usepackage{graphics}
\usepackage{amsmath}
\usepackage{amssymb}
\usepackage{mathrsfs}
\usepackage{amsthm}
\usepackage{xcolor}
\usepackage{accents}
\usepackage{enumerate}
\usepackage{url}
\usepackage{tikz-cd}
\usepackage{marginnote}
\usepackage{hyperref}

\usepackage{newpxtext} \usepackage[euler-digits]{eulervm}

\theoremstyle{plain}
\newtheorem{theorem}{\sc Theorem}[section]
\newtheorem{prop}[theorem]{\sc Proposition}
\newtheorem{lem}[theorem]{\sc Lemma}

\theoremstyle{definition}
\newtheorem{defn}[theorem]{\sc Definition}
\newtheorem{rem}[theorem]{\sc Remark}

\DeclareMathOperator{\im}{im}

\DeclareMathOperator{\R}{\mathbb{R}}
\newcommand{\abs}[1]{\left\lvert#1\right\rvert}
\newcommand{\norm}[1]{\left\lVert#1\right\rVert}

\title{Order of Contact and Ruled Submanifolds}
\date{\today}
\author{Igor Uljarevi\'{c}\\\\ University of Belgrade\\ Faculty of Mathematics\\\\ igoru@matf.bg.ac.rs}

\usepackage{biblatex}
\begin{document}
\maketitle{}

\begin{abstract}
We prove a generalization of the Monge-Cayley-Salmon theorem on osculation and ruled submanifolds using geometric measure theory.
\end{abstract}

\section{Introduction}

Analytic surfaces in $\R^3$ have the following remarkable property, that played a key role in the proof of the Erd\H{o}s distinct distances problem in dimension two by Guth and Katz \cite{guth2015erdHos}.

\begin{theorem}[Monge, Cayley, Salmon]\label{thm:MCS}
Let $M\subset \R^3$ be a proper $3$-dimensional analytic surface. Assume there exists a smooth family $\ell_x,x\in M$ of lines in $\R^3$ such that, for all $x\in M,$ $\ell_x$ and $M$ have a contact of order $3$ at $x.$ Then, $\ell_x\subset M$ for all $x\in M.$
\end{theorem}

A proof of Theorem~\ref{thm:MCS} can be found in \cite{salmon1865treatise}. In \cite{guth2018algebraic}, Guth and Zahl proved a version of Theorem~\ref{thm:MCS} for an arbitrary field instead of $\R.$ For a detailed exposition on the Monge-Cayley-Salmon theorem and its relation to the Erd\H{o}s distinct distances problem, we refer to \cite{katz2014flecnode} and \cite{tao2014blog}. The aim of the present paper is to prove the following generalization of Theorem~\ref{thm:MCS}. Along the line, we present  a novel elementary proof of the Monge-Cayley-Salmon theorem.

\begin{defn}\label{defn:classk}
A curve $\Gamma\subset \R^n$ is said to be of class $k\in \mathbb{N}$ if it can be parametrized by a map $\R\to\R^n$ whose coordinates are polynomial functions of degree at most $k.$
\end{defn}

In the terminology of Definition~\ref{defn:classk}, the lines are curves of class 1.

\begin{theorem}\label{thm:generalization}
Let $M\subset \R^n$ be a proper $m$-dimensional analytic submanifold. Assume there exists a smooth family $\Gamma_x,x\in M$ of class-$k$ curves in $\R^n$ such that, for all $x\in M,$ $\Gamma_x$ and $M$ have a contact of order $k(m+1)$ at $x.$ Then, $\Gamma_x\subset M$ for all $x\in M.$
\end{theorem}

The proof of Theorem~\ref{thm:generalization} for $k=1$ and for analytic submanifolds of $\mathbb{C}^n$ or $\mathbb{C}P^n$ can be found in \cite{landsberg1999linear}. That proof uses techniques of algebraic geometry. 

The main idea of our proof is to consider the $(m+1)$-dimensional volume swept by $M$ as each point $x$ of $M$ moves along $\Gamma_x.$ It turns out (see Proposition~\ref{prop:contained} on page~\pageref{prop:contained}) that this volume is equal to 0 precisely when $\Gamma_x\subset M$ for all $x\in M.$ The order-of-contact condition, on the other hand, implies that the rate at which the volume is swept is sufficiently slow (see Proposition~\ref{prop:analyticlimit0} on page~\pageref{prop:analyticlimit0}). What bridges these two facts (the vanishing volume and the volume being swept at a sufficiently slow rate) is a result of a Weyl-tube-formula type (Proposition~\ref{prop:notarbitrary} on page \pageref{prop:notarbitrary}). Now, we sketch this step in the case of a hypersurface in $\R^n.$ The volume swept by the hypersurface is a polynomial in the time-variable $t$ of degree at most $k(m+1)=kn.$ If a polynomial of degree at most $kn$ grows slower than $t^{kn}$ as we approach 0, then it vanishes identically. 

\subsection*{Acknowledgments}

I would like to thank Filip Mori\'{c} for suggesting this project. I am grateful to Darko Milinkovi\'{c} for many useful discussions. This work was supported by the Ministry of Education, Science, and Technological development, grant number 174034.

\section{Preliminaries}

\subsection{The nearest point map}

\begin{defn}
Let $M\subset \R^n$ be a smooth submanifold, and let $N$ be a normal tubular neighbourhood of $M.$ The nearest point map $r:N\to M$ is the map that sends each point of $N$ to the unique nearest point in $M.$
\end{defn}

The nearest point map of a smooth submanifold is smooth \cite[page~109]{hirsch2012differential}. If the submanifold is analytic, then the nearest point map is analytic as well \cite[page~240]{federer2014geometric} .

\subsection{Order of contact}

\begin{defn}
Two smooth curves $\gamma_1,\gamma_2:\R\to  \R^n$ are said to have a contact of order $k\in\mathbb{N}\cup\{0\}$ at a point $t_0\in\R$ if
\[(\forall j\in\{0,1,\ldots,k\})\quad \gamma_1^{(j)}(t_0)=\gamma_2^{(j)}(t_0).\]
\end{defn}

\begin{defn}
A smooth curve $\gamma_1:\R\to\R^n$ is said to have a contact of order $k\in\mathbb{N}\cup \{0\}$ with a smooth submanifold $M\subset \R^n$ at a point $t_0\in \R$ if there exists a smooth curve $\gamma_2:\R\to M$ such that the curves $\gamma_1$ and $\gamma_2$ have a contact of order $k$ at $t_0.$
\end{defn}

\begin{defn}
Two submanifolds $M_1,M_2\subset \R^n$ are said to have a contact of order $k\in\mathbb{N}\cup\{0\}$ at a point $p\in M_1\cap M_2$ if for every smooth curve $\gamma_1:\R\to M_1$ such that $\gamma_1(0)=p$ there exists a smooth curve $\gamma_2:\R\to M_2$ such that $\gamma_1$ and $\gamma_2$ have a contact of order $k$ at 0.
\end{defn}

In the following lemma (and in the rest of the paper), we denote by $d(x, M)$ the distance between a point $x\in \R^n$ and a subset $M\subset\R^n, $ i.e.
\[d(x,M):=\inf\left\{\norm{x-y}\quad|\quad y\in M \right\}.\]

\begin{lem}\label{lem:uniform}
Let $M$ be a submanifold of $\R^n,$ let $I$ be an open interval containing 0, and let $\phi_t:M\to \R^n, t\in I$ be a smooth family of embeddings. Assume, for all $x\in M,$ the curve $t\mapsto \phi_t(x)$ has a contact of order $k\in\mathbb{N}\cup\{0\}$ with $M$ at $t=0.$ Then,
\[\lim_{t\to 0}\frac{d(\phi_t(x),M)}{t^k}=0\]
uniformly on compact subsets of $M.$
\end{lem}
\begin{proof}
Let $N$ be a normal tubular neighbourhood of $M,$ and let $r:N\to M$ be the nearest point map. Let $K\subset M$ be an arbitrary compact subset, and let $\delta>0$ be such that $[-\delta,\delta]\subset I$ and such that $\phi_t(x)\in N$ for $x\in K$ and $t\in[-\delta,\delta].$ Denote
\[C:=\max_{x\in K, \abs{t}\leqslant \delta}\norm{\partial_t^{k+1}(\phi_t(x)-r\circ \phi_t(x))}.\]
Fix $x\in M,$ and denote by $\gamma=(\gamma_1,\ldots,\gamma_n): I\to \R^n$ the curve defined by $\gamma(t)=\phi_t(x)-r(\phi_t(x)).$ Since $\phi_t(x)$ has a contact of order $k$ with $M,$ the derivatives of $\gamma$ up to order $k$ are equal to 0. The Taylor approximation implies
\[ \abs{\gamma_i(t)}\leqslant \frac{t^{k+1}}{(k+1)!} \cdot \max_{\abs{t}\leqslant\delta}\abs{\partial_t^{k+1}\gamma_i(t)} \leqslant t^{k+1}\cdot \frac{C}{(k+1)!},\]
for all $i\in \{1,\ldots, n\}$ and $t\in(-\delta,\delta).$ Hence
\[\frac{d(\phi_t(x),M)}{t^k}=\frac{\norm{\gamma(t)}}{t^k}\leqslant t\cdot\frac{\sqrt{n}\cdot C}{(k+1)!},\]
for all $x\in K$ and $t\in (-\delta,\delta).$ This finishes the proof.
\end{proof}

\subsection{Volume of a $C^1$ map}

The purpose of this section is to introduce the notion of a volume $Vol(\phi)$ of a $C^1$ map $\phi:M\to N$ from a smooth manifold $M$ to a Riemannian manifold $(N,g).$ Intuitively, the volume of a $C^1$ map $\phi:M\to N$ is the $(\dim M)-$dimensional volume swept by $\phi$ in $N.$ We will, actually, formally define only the volume of a $C^1$ map form an open subset of $\R^m$ to a Riemannian manifold. This definition extends to the general case via the standard trick which uses a collection of charts and a subordinate partition of unity.

\begin{defn}
Let $U\subset \R^m$ be an open subset, let $(N,g)$ be a Riemannian manifold, and let $\phi: U\to N$ be a $C^1$ map. The \emph{volume} of the map $\phi$ is defined by
\[Vol(\phi):=\int_{U}\norm{\partial_1\phi(x) \wedge \cdots \wedge \partial_m\phi(x) }dx.\]
\end{defn}

\begin{rem}
We found it convenient to use exterior algebra $\bigwedge^m T_{\phi(x)}N$ to express the volume element. In our conventions, if $e_1,\ldots, e_n$ is an orthonormal basis of $T_{\phi(x)}N,$ then $e_{i_1}\wedge\cdots\wedge e_{i_m}, 1\leqslant i_1<\cdots<i_m\leqslant n$ is an orthonormal basis of $\bigwedge^m T_{\phi(x)}N.$ Alternatively,   $\norm{\partial_1\phi(x) \wedge \cdots \wedge \partial_m\phi(x) }$ can be written as
\[\norm{\partial_1\phi(x) \wedge \cdots \wedge \partial_m\phi(x) }= \sqrt{\det\left[ g(\partial_i\phi(x), \partial_j\phi(x)) \right]_{i,j}}.\]
\end{rem}

\begin{lem}\label{lem:coordinates}
Let $\psi:V\to U$ be a $C^1$ diffeomorphism between two $m$-dimensional manifolds, let $(N,g)$ be a Riemannian manifold, and let $\phi:U\to N$ be a $C^1$ map. Then, $Vol(\phi\circ\psi)=Vol(\phi).$ 
\end{lem}
\begin{proof}
Without loss of generality, assume $U$ and $V$ are two open subsets of $\R^m.$ For a linear map $A:W_1\to W_2$ (from a vector space $W_1$ to a vector space $W_2$), denote by $\bigwedge ^k A :\bigwedge^k W_1\to \bigwedge^k W_2$ the linear map defined by
\[(\forall v_1,\ldots v_k\in W_1)\quad \left({\bigwedge}^k A\right)(v_1\wedge\cdots\wedge v_m):=(A v_1)\wedge\cdots\wedge (A v_k). \]
The lemma follows from the following sequence of equalities
\begingroup
\allowdisplaybreaks
\begin{align*}
    Vol(\phi\circ\psi)&= \int_V\norm{ \partial_1(\phi\circ\psi) \wedge \cdots \wedge \partial_m(\phi\circ\psi) }dx\\
    &= \int_V \norm{ (D\phi(\psi(x))\partial_1\psi(x)) \wedge\cdots \wedge (D\phi(\psi(x))\partial_m\psi(x))} dx\\
    &= \int_V \norm{{\bigwedge}^k(D\phi(\psi(x))) \partial_1\psi(x) \wedge \cdots \wedge \partial_m \psi(x)} dx\\
    &= \int_V \norm{ {\bigwedge}^k(D\phi(\psi(x))) \det{D\psi(x)} e_1 \wedge \cdots \wedge e_m } dx\\
    &= \int_V \norm{ {\bigwedge}^k(D\phi(\psi(x))) e_1 \wedge \cdots \wedge e_m }\cdot \abs{\det{D\psi(x)}}dx\\
    &= \int_V \norm{D\phi(\psi(x))e_1 \wedge \cdots \wedge D\phi(\psi(x))e_m }\cdot\abs{\det{D\psi(x)}}dx\\
    &= \int_V \norm{\partial_1\phi(\psi(x)) \wedge \cdots \wedge \partial_m\phi(\psi(x))}\cdot\abs{\det{D\psi(x)}}dx\\
    &= \int_U \norm{\partial_1\phi(y) \wedge \cdots \wedge \partial_m\phi(y)}dy\\
    &= Vol(\phi).
\end{align*}
\endgroup
Here, $e_1,\ldots, e_m$ stands for the standard basis of $\R^m.$
\end{proof}

\section{Family of embeddings and the swept volume}

In this section, we consider the maps of the form
\begin{align*}
&\phi: M\times \R\to\R^n,\\
&\phi(x,t):= x+ t\cdot v_1(x)+\cdots t^k\cdot v_k(x),
\end{align*}
where $M$ is an $m$-dimensional submanifold, and $v_i:M\to\R^n$ is a smooth map for $i\in\{1,\ldots,k\}.$ Proposition~\ref{prop:notarbitrary} proves that the rate of growth of $Vol\left(\left.\phi\right|_{M\times(-t,t)}\right)$ at 0 cannot be arbitrary. More precisely, it shows that $Vol\left(\left.\phi\right|_{M\times(-t,t)}\right)=o(t^{k(m+1)})$ implies $Vol\left(\left.\phi\right|_{M\times(-t,t)}\right)\equiv 0.$ Proposition~\ref{prop:contained} is a general statement about a smooth 1-parameter family of embeddings with vanishing volume. Proposition~\ref{prop:analyticlimit0} relates the order of contact with the growth rate of $Vol\left(\left.\phi\right|_{M\times(-t,t)}\right)$ at 0. In this proposition, it is assumed that $M$ is an analytic submanifold of $\R^n.$

\begin{prop}\label{prop:notarbitrary}
Let $M\subset\R^n$ be an $m$-dimensional submanifold, and let $v_i:M\to \R^n,$  $i\in\{1,\ldots,k\}$ be smooth maps. Denote by $\phi:M\times\R\to \R^n$ the map defined by
\[\phi(x,t):=x+t\cdot v_1(x)+\cdots+ t^k\cdot v_k(x).\]
If
\[ \lim_{t\to 0} \frac{Vol\left(\left.\phi\right|_{M\times(-t,t)}\right)}{t^{k(m+1)}}=0, \]
then $Vol(\phi)=0.$
\end{prop}
\begin{proof}
Denote $d:= k(m+1)-1.$ Let $\alpha: U \to \R^m,$ where $U\subset \R^m$ is open, be a parametrization of a subset of $M.$ Denote by $\psi: U\times \R\to\R^n$ the map defined by $\psi(x,t):=\phi(\alpha(x),t).$ The map $\psi$ is equal to the composition of the restriction $\left.\phi\right|_{\alpha(U)\times\R}$ with the diffeomorphism
\[U\times\R\to\alpha(U)\times \R\quad:\quad (x,t)\mapsto (\alpha(x),t).\]
Hence
\[Vol\left(\left.\psi\right|_{U\times(-t,t)}\right)=Vol\left(\left.\phi\right|_{\alpha(U)\times(-t,t)}\right)\leqslant Vol\left(\left.\phi\right|_{M\times(-t,t)}\right), \]
and, consequently, 
\[\lim_{t\to 0}\frac{Vol\left(\left.\psi\right|_{U\times(-t,t)}\right)}{t^{d+1}}=0.\]
The volume element $\norm{ \partial_1\psi(x,t)\wedge\cdots\wedge \partial_m\psi(x,t)\wedge\partial_t\psi(x,t)}$ is equal to
\[\norm{\bigwedge_{i=1}^m\left( \partial_i\alpha(x) + \sum_{j=1}^{k} t^j\cdot\partial_i (v_j\circ\alpha)(x)\right)\wedge\left(\sum_{j=1}^{k} j\cdot t^{j-1}\cdot (v_j\circ\alpha)(x)\right)}.\]
After developing the expression above by distributive law and after applying the Pythagorean theorem, the volume element transforms into the form
\[ \sqrt{A_1(x,t)^2+\cdots+A_\ell(x,t)^2}\]
where $A_1,\ldots, A_\ell$ are polynomials in $t$ of degree at most $d$ whose coefficients are smooth functions $U\to\R$ in $x$-variable, and $\ell:=\binom{n}{m+1}.$ Assume there exists $j\in\{1,\ldots, \ell\}$ such that $A_j(x,t)$ is not identically equal to 0. Then,
\[A_j(x,t)=a_0(x)+a_1(x)\cdot t+\cdots+ a_d(x)\cdot t^d, \]
where $a_1,\ldots,a_d:U\to \R$ are smooth functions that are not all identically equal to 0. Let $b\in\{0,\ldots,d\}$ be the smallest index such that $a_b$ is not identically equal to 0. There exists $\varepsilon>0$ such that
\[ \frac{1}{2}\int_U\abs{a_b(x)}dx\cdot \abs{t}^b\geqslant \sum_{i=b+1}^d \int_U \abs{a_i(x)}dx \cdot \abs{t}^i, \]
for $\abs{t}<\varepsilon.$ Hence, for $\abs{t}<\varepsilon,$
\begingroup
\allowdisplaybreaks
\begin{gather*}
    \int_U\sqrt{A_1(x,t)^2+\cdots+A_\ell(x,t)^2} dx\\
    \rotatebox[origin=c]{-90}{$\geqslant$}\\
    \int_U\abs{A_j(x,t)}dx\\
    \rotatebox[origin=c]{-90}{$\geqslant$}\\
    \int_U\abs{a_b(x)}dx\cdot\abs{t}^b- \sum_{i=b+1}^d\int_U\abs{a_i(x)}dx\cdot\abs{t}^i\\
    \rotatebox[origin=c]{-90}{$\geqslant$}\\
    \frac{1}{2}\int_U\abs{a_b(x)}dx\cdot\abs{t}^b
\end{gather*}
\endgroup
This further implies
\begin{align*}
    0&=\lim_{t\to 0}\frac{Vol\left(\left.\psi\right|_{U\times(-t,t)}\right)}{t^{d+1}}\\
    &\geqslant \lim_{t\to 0}t^{-(d+1)}\cdot \int_{-t}^t \frac{1}{2}\int_U\abs{a_b(x)}dx\cdot \abs{s}^bds\\
    &= \lim_{t\to 0}t^{-(d+1)}\cdot \int_{0}^t \int_U\abs{a_b(x)}dx\cdot \abs{s}^bds\\
    &=\lim_{t\to 0} t^{-(d+1)}\cdot  \frac{t^{d+1}}{d+1} \cdot \int_U\abs{a_d(x)}dx\\
    &= \frac{1}{d+1}\int_U\abs{a_b(x)}dx.
\end{align*}
The continuity of $a_b$ now implies $a_b(x)=0$ for all $x\in U.$ Contradiction! Therefore 
\[A_1(x,t)=\cdots=A_\ell(x,t)=0,\]
for all $x\in U$ and $t\in\R,$ and, consequently, 
\[Vol\left(\left.\phi\right|_{\alpha(U)\times \R}\right)=Vol(\psi)=0.\]
This holds for all charts $\alpha$ of $M.$ Hence $Vol(\phi)=0.$
\end{proof}

\begin{prop}\label{prop:contained}
Let $M$ be an $m$-dimensional manifold, let $I$ be an open interval containing 0, and let $\phi_t:M\to\R^n$ be a smooth family of embeddings. Assume $Vol(\phi: M\times I\to \R^n)=0.$ Then, for all $x\in M,$ there exists $\varepsilon=\varepsilon(x)>0$ such that 
\[(\forall t\in (-\varepsilon,\varepsilon))\quad \phi_t(x)\in \phi_0(M).\]
\end{prop}
\begin{proof}
 Suppose there exists $(x_0,t_0)\in M\times I$ such that $D\phi(x_0,t_0)$ is of rank $m+1.$ Then, there exists a neighbourhood $U\subset M\times I$ of $(x_0,t_0)$ such that $\left.\phi\right|_{U}$ is an embedding. The volume of an embedding is positive. Hence
 \[Vol(\phi)\geqslant Vol(\left.\phi\right|_{U})>0.\]
 This contradicts $Vol{\phi}=0.$ Therefore 
\[(\forall (x,t)\in M\times I)\quad \operatorname{rank} D\phi(x,t)\leqslant m.\]
Since $\phi_t$ is a family of embeddings, the rank of $D\phi(x,t)$ is less than $m+1$ if, and only if, 
\[\partial_t\phi_t(x)\in \im D\phi_t(x).\]
Denote by $Y_t$ the smooth vector field on $M$ defined by 
\[\partial_t\phi_t(x)= D\phi_t(x) Y_t(x).\]
Denote by $\psi$ the (locally defined) flow of the vector field $-Y_t.$ Fix $x\in M.$ Let $O$ be a neighbourhood of $x$ and let $\varepsilon>0$ be such that $\psi_t(y)$ is well defined for $y\in O$ and $t\in(-\varepsilon,\varepsilon).$  Let $\delta>0$ be such that $\psi_t^{-1}(x)\in O $ for all $t\in (-\delta,\delta).$ Since
\[\begin{split}
\frac{\partial}{\partial t}(\phi_t(\psi_t(x)))&= D\phi_t(\psi_t(x))Y_t(\psi_t(x))+ D\phi_t(\psi_t(x))\partial_t\psi_t(x)\\
&= D\phi_t(\psi_t(x))Y_t(\psi_t(x))+ D\phi_t(\psi_t(x))(-Y_t(\psi_t(x)))\\
&=0,
\end{split}\]
for all $y\in O$ and $t\in(-\varepsilon, \varepsilon),$ we have $\phi_t(\psi_t(y))=\phi_0(y).$  By substituting $y=\psi_t^{-1}(x),$ one gets
\[(\forall t\in (-\delta,\delta))\quad \phi_t(x)=\phi_0(\psi_t^{-1}(x)))\in\phi_0(M).\]
\end{proof}

\begin{prop}\label{prop:analyticlimit0}
Let $M\subset \R^n$ be an analytic submanifold, and let $\phi:M\times\R\to \R^n$ be a smooth map such that
\begin{itemize}
    \item $\phi(x,0)=x,$ for all $x\in M,$
    \item $\phi(x,t)=x,$ for $t\in \R$ and for $x$ outside of a compact set,
    \item the curve $t\mapsto \phi(x,t)$ is analytic and has a contact of order $k\in \mathbb{N}$ with $M$ at $t=0$ for all $x\in M.$
\end{itemize}
Then,
\[\lim_{t\to 0} \frac{Vol\left(\left.\phi\right|_{M\times(-t,t)}\right)}{t^{k}}=0.\]
\end{prop}
\begin{proof}
Without loss of generality, assume that $M$ is covered by a single chart $\alpha:U\subset\R^m\to M.$ Let $N$ be a normal tubular neighbourhood of $M,$ and let $r:N\to M$ be the nearest point map. Since $\phi_t(x)$ is $t$-independent for $x$ outside of a compact set (and since $\phi_0(x)\in M$), there exists $\delta>0$ such that $\phi(M\times (-\delta,\delta))\subset N.$ Let $\delta_1\in (0,\delta)$ be such that $\phi_t: M\to \R^n$ is an embedding for  all $t\in [-\delta_1,\delta_1].$ Such $\delta_1$ exists because the set of embeddings $M\to \R^n$ is open \cite[Theorem~1.4]{hirsch2012differential}. Let $Y_t, t\in[-\delta_1,\delta_1]$ be the vector field on $M$ defined by
\[D\phi_t(x) Y_t(x)=-dr(\partial_t\phi_t(x)).\]
Denote by $\theta_t:M\to M$ the flow of the vector field $Y_t.$ Let $\psi_t:M\to \R^n$ be the smooth family of embeddings defined by $\psi_t(x):=\phi_t(\theta_t(x)).$  For $t\in[-\delta_1,\delta_1],$ the following holds
\[\phi_t(M)=\psi_t(M),\quad Vol\left(\left.\phi\right|_{M\times(-t,t)}\right)=Vol\left(\left.\psi\right|_{M\times(-t,t)}\right).\]
There exists a compact set $K\subset U$ such that $\psi_t(x)=x$ for $x\in \alpha(U\setminus K)$ and $t\in[-\delta_1,\delta_1].$ Denote
\[ C:=\max _{x\in K, \abs{t}\leqslant\delta_1} \norm{\partial_1(\psi_t\circ\alpha)(x)\wedge\cdots\wedge\partial_m(\psi_t\circ\alpha)(x)}.\]
Since
\[\norm{v_1\wedge\cdots\wedge v_\ell}\leqslant \norm{v_1}\cdot\norm{v_2\wedge\cdots\wedge v_\ell},\]
we get
\begin{align*}
    Vol&(\left.\phi\right|_{M\times(-t,t)})=Vol(\left.\psi\right|_{M\times(-t,t)})\\
    &=\int_U\int_{-t}^t\norm{\partial_1(\psi_s\circ\alpha)(x)\wedge\cdots\wedge\partial_m(\psi_s\circ\alpha)(x)\wedge\partial_s\psi_s\circ\alpha(x)}dsdx\\
    &\leqslant\int_U\int_{-t}^t\norm{\partial_s\psi_s\circ\alpha(x)}\cdot\norm{\partial_1(\psi_s\circ\alpha)(x)\wedge\cdots\wedge\partial_m(\psi_s\circ\alpha)(x)}dsdx\\
    &=\int_K\int_{-t}^t\norm{\partial_s\psi_s\circ\alpha(x)}\cdot\norm{\partial_1(\psi_s\circ\alpha)(x)\wedge\cdots\wedge\partial_m(\psi_s\circ\alpha)(x)}dsdx\\
    &\leqslant C\cdot \int_K\int_{-t}^t\norm{\partial_s\psi_s\circ\alpha(x)}ds,
\end{align*}
for $t\in[0,\delta_1).$ It is enough to prove 
\begin{equation}\label{eq:lim}
\lim_{t\to 0}\frac{1}{t^k}\cdot\int_{-t}^t\norm{\partial_s\psi_s(x)}ds=0.
\end{equation}
Since
\begin{align*}
    \partial_t\psi_t(x)&= (\partial_t\phi_t)(\theta_t(x))+ D\phi_t(\theta_t(x))\partial_t\theta_t(x)\\
    &=(\partial_t\phi_t)(\theta_t(x))+ D\phi_t(\theta_t(x))Y_t(\theta_t(x))\\
    &=(\partial_t\phi_t)(\theta_t(x))- dr(\partial_t\phi_t(\theta_t(x)))\\
    &\in \ker dr,
\end{align*}
we have 
\[ \partial_t r(\psi_t(x)) = dr(\psi_t(x))\partial_t\psi_t(x)=0, \]
and, consequently, $r(\psi_t(x))=x$ for $t\in(-\delta_1,\delta_1).$ By Lemma~\ref{lem:analyticmonotone} below, for $x\in M,$ there exists $\varepsilon_x>0$ such that the coordinates of $r(\psi_t(x))-\psi_t(x)=x-\psi_t(x)$ are monotone (with respect to $t$) for $t\in [-\varepsilon_x,0]$ and for $t\in [0,\varepsilon_x].$ Lemma~\ref{lem:monotonecoordinates} below implies
\begin{align*}
    \int_{-t}^t \norm{\partial_s\psi_s(x)}ds&= \int_{-t}^t \norm{\partial_s(x-\psi_s(x))}ds\\
    &=\int_{-t}^0 \norm{\partial_s(x-\psi_s(x))}ds + \int_{0}^t \norm{\partial_s(x-\psi_s(x))}ds\\
    &\leqslant n\cdot\left( \norm{\psi_{-t}(x)-x} + \norm{\psi_t(x)-x} \right)\\
    &= n\cdot \left( d(\psi_{-t}(x), M) + d(\psi_{t}(x), M)  \right)\\
    &\leqslant n\cdot\left( \sup_{y\in M}d(\psi_{-t}(y), M) + \sup_{y\in M}d(\psi_{t}(y), M) \right)\\
    &=n\cdot\left( \sup_{y\in M}d(\phi_{-t}(y), M) + \sup_{y\in M}d(\phi_{t}(y), M) \right)
\end{align*}
for $t\in[0,\varepsilon_x].$ Hence (by Lemma~\ref{lem:uniform}) \eqref{eq:lim} holds, and the proof is finished.
\end{proof}

\begin{lem}\label{lem:analyticmonotone}
Let $M\subset \R^n$ be an analytic submanifold, and let $\gamma:\R\to\R^n$ be an analytic curve such that $\gamma(0)\in M.$ Denote by $r:N\to M$ the nearest point map defined in a normal tubular neighbourhood $N$ of $M.$ Then, there exists $\varepsilon>0$ such that the coordinates of the function
\[[-\varepsilon,\varepsilon]\to \R^n\quad:\quad t\mapsto r(\gamma(t))-\gamma(t) \]
are monotone (not necessarily strictly) on $[-\varepsilon,0]$ and $[0,\varepsilon].$
\end{lem}
\begin{proof}
Since $M$ is an analytic submanifold of $\R^n,$ the nearest point map is analytic \cite[page~240]{federer2014geometric}. The set $\gamma^{-1}(N)$ is open. Hence there exists $\delta>0$ such that $\gamma(t)\in N$ for all $t\in(-\delta,\delta).$ Denote by
\[f=(f_1,\ldots, f_n):(-\delta,\delta)\to\R^n\]
the analytic map defined by $f(t)=r(\gamma(t))-\gamma(t).$ Fix $j\in\{1,\ldots,n\}.$ If $f_j^{(k)}(0)=0$ for all $k\in\mathbb{N},$ then (since $f_j$ is analytic) there exists $\varepsilon>0$ such that $f_j(t)=0$ for $t\in(-\varepsilon,\varepsilon).$ Consequently, $f_j$ is monotone on $(-\varepsilon,\varepsilon).$
Assume, now, there exists $k\in \mathbb{N}$ such that $f_j^{(k)}(0)\not=0,$ and such that $f_j^{(i)}(0)=0$ for all $i\in\{1,\ldots,k-1\}.$ The Taylor approximation for $f'_j$ implies
\begin{align*}
    f_j'(t) &= \frac{f_j^{(k)}(0)}{(k-1)!}\cdot t^{k-1} + \frac{f_j^{(k+1)}(c_t)}{k!}\cdot t^{k}\\
    &=\frac{t^{k-1}}{(k-1)!}\left( f_j^{(k)}(0)+\frac{1}{k}\cdot f_j^{(k+1)}(c_t)\cdot t \right),
\end{align*} 
for $t\in (-\delta,\delta)$ and for some $c_t$ between 0 and $t.$ Since $\frac{1}{k}\cdot f_j^{(k+1)}$ is a bounded function on $\left[-\frac{\delta}{2},\frac{\delta}{2}\right],$ there exists $\varepsilon\in\left(0,\delta\right)$ such that the function $f'_j$ does not change the sign on intervals $(-\varepsilon,0)$ and $(0,\varepsilon).$ Hence $f_j$ is monotone on $[-\varepsilon,0]$ and $[0,\varepsilon].$
\end{proof}

\begin{lem}\label{lem:monotonecoordinates}
Let $[a,b]\subset\R$ be a compact interval, and let $\gamma=(\gamma_1,\ldots,\gamma_n):[a,b]\to\R^n$ be a $C^1$ curve such that $\gamma_i:[a,b]\to\R$ is monotone for all $i\in\{1,\ldots,n\}.$ Then,
\[\operatorname{length}(\gamma)\leqslant n\cdot \norm{\gamma(a)-\gamma(b)}.\]
\end{lem}
\begin{proof}
Since
\[\sqrt{\sum_{j}\left(\gamma_j'(t)\right)^2}\leqslant \sqrt{n}\cdot \max_{j}\abs{\gamma_j'(t)}\leqslant \sqrt{n}\cdot\sum_j\abs{\gamma_j'(t)},\]
the length of $\gamma$ is bounded by
\begin{align*}
    \operatorname{length}(\gamma)&\leqslant \sqrt{n}\cdot \sum_j\int_a^b\abs{\gamma_j'(t)}dt\\
    &=\sqrt{n}\cdot\sum_j\abs{\int_a^b \gamma_j'(t)dt}\\
    &=\sqrt{n}\cdot \sum_j\abs{\gamma_j(b)-\gamma_j(a)}\\
    &\leqslant n^{\frac{3}{2}}\cdot\sqrt{\frac{1}{n}\cdot \sum_j\abs{\gamma_j(b)-\gamma_j(a)}^2}\\
    &=n\cdot\norm{\gamma(b)-\gamma(a)}.
\end{align*}
In the sequence of inequalities above, we used 
\[\abs{\int_a^b\gamma'_j(t)dt}=\int_a^b\abs{\gamma'_j(t)}dt\]
(which holds because $\gamma_j$ is monotone) and the Cauchy-Schwarz inequality.
\end{proof}

\section{Proof of the main theorem}

\begin{proof}[Proof of Theorem~\ref{thm:generalization}]
Let $p$ be an arbitrary point in $M,$ and let $v_1,\ldots, v_k:M\to \R^n$ be smooth compactly supported maps such that
\[x + t\cdot v_1(x) + \cdots + t^k\cdot v_k(x)\in \Gamma_x\]
for all $x\in M,$ and such that
\[t\mapsto p + t\cdot v_1(p) + \cdots + t^k\cdot v_k(p)\]
is a parametrization of $\Gamma_p.$ Denote by $\phi_t:M\to\R^n, t\in\R$ the family of smooth maps defined by
\[\phi_t(x):=x + t\cdot v_1(x) + \cdots + t^k\cdot v_k(x).\]
Proposition~\ref{prop:analyticlimit0} implies
\begin{equation}\label{eq:lim0}
    \lim_{t\to 0}\frac{Vol\left(\left.\phi\right|_{M\times(-t,t)}\right)}{t^{k(m+1)}}=0.
\end{equation}
There exists $\varepsilon>0,$ such that $\phi_t,t\in (-\varepsilon,\varepsilon)$ is a smooth family of embeddings (see \cite[Theorem~1.4]{hirsch2012differential}). Therefore, due to Proposition~\ref{prop:notarbitrary}, Proposition~\ref{prop:contained}, and \eqref{eq:lim0}, $\phi_t(x)\in M$ for $\abs{t}$ small enough. In particular, there exists an open segment $I_p$  of $\Gamma_p$ such that $p\in I_p\subset M.$ Since $M$ is proper, and  since $M$ and $\Gamma_p$ are analytic, the identity theorem for analytic functions \cite[Corollary~1.2.7]{krantz2002primer} implies $\Gamma_p\subset M.$
\end{proof}

\printbibliography

\end{document}